\documentclass[12pt]{amsart}
\usepackage{a4}
\usepackage{amssymb}

\makeatletter
\@addtoreset{equation}{section}
\makeatother

\usepackage{color}
\usepackage[mathscr]{eucal}
\usepackage{amsmath,amsthm,amssymb}
\usepackage{mathrsfs}
\usepackage{enumerate}
\usepackage{bm}
\usepackage{graphicx}
\usepackage{verbatim}
\usepackage{wrapfig}
\usepackage{ascmac}
\usepackage{multicol}
\usepackage{mathtools}
\usepackage{multirow}
\usepackage{diagbox}
\usepackage{here}

\usepackage{hyperref}

\newtheorem{Prop}{Proposition}[section]
\newtheorem{Thm}[Prop]{Theorem}
\newtheorem{Lem}[Prop]{Lemma}

\theoremstyle{definition}
\newtheorem{Def}[Prop]{Definition}

\newtheorem{Rem}[Prop]{Remark}

\newcommand{\Z}{\mathbb{Z}}

\newcommand{\ad}{ad}

\newcommand{\Der}{\mathrm{Der}}
\newcommand{\Path}{\mathrm{Path}}
\newcommand{\fr}[1]{\mathfrak{#1}}

\title[Nilpotent Lie algebras from parabolic subalgebras and quivers]{On nilpotent Lie algebras arising from parabolic subalgebras and quivers --- three-step cases}

\author{Fumika Mizoguchi}
\address[F.~Mizoguchi]{RIKEN Center for Advanced Intelligence Project} 
\email{fumika.mizoguchi@riken.jp}

\date{}

\makeatletter
\@namedef{subjclassname@2020}{\textup{2020} Mathematics Subject Classification}
\makeatother
\subjclass[2020]{}

\keywords{nilpotent Lie algebras, parabolic subalgebras} 

\begin{document} 

\maketitle 

\begin{abstract} 
	Nilpotent Lie groups are important subjects for studying spacial left-invariant geometric structures. However, nilpotent Lie groups are impossible to classify, and even the two-step case has not been classified. There are several classes of examples of nilpotent Lie algebras that arise with higher nilpotency steps. In this paper, we focus on two of these methods. The first one is the nilpotent parts of parabolic subalgebras of semisimple Lie algebras. The second one is the nilpotent Lie algebras obtained by quivers. We list and compare nilpotent Lie algebras obtained by these two methods that are three-step and have dimension $\leq 10$. 
\end{abstract}

\section{Introduction}
	Nilpotent Lie algebras play important role in the study of special left-invariant geometric structures.  In fact, nilpotent Lie groups provide various examples that admit nontrivial Ricci soliton metrics and non-flat Ricci-flat pseudo-Riemanninan metrics. However, nilpotent Lie groups are impossible to classify, and even the two-step case has not been classified. As the step increases, the complexity grows, making general theory much harder to develop. Based on the above, it is difficult to construct higher-step easy-to-handle examples. 
	
	There are several classes of examples of nilpotent Lie algebras that arise with higher nilpotency steps. One of them is nilpotent Lie algebras obtained by quivers. A quiver is a directed graph where loops and multiple arrows between two vertices are allowed. Utilizing the concept of paths within quiver, in our work, we defined nilpotent Lie algebras obtained by finite quivers without cycles (\cite{MT}). For a finite quiver without cycles, there exists the maximum of length of paths which is called the length of a quiver. Then, for a quiver $Q$ of length $m$, the obtained Lie algebra $\fr{n}_Q$ is an $m$-step nilpotent Lie algebra. Therefore, from quivers we can construct nilpotent Lie algebras with arbitrarily high nilpotency steps. 
	
	Next, other examples are nilpotent parts of parabolic subalgebras. They are examples that arise in higher-step nilpotent Lie algebras. Let $\fr{q}_{\lambda'}$ be the parabolic subalgebra of a semisimple Lie algebra $\fr{g}$ corresponding to a subset $\Lambda'$ of a simple roots $\Lambda$ of the restricted root system of $\fr{g}$. Then $\fr{q}_{\Lambda'}$ has a decomposition $\fr{q}_{\Lambda'}=\fr{m}_{\Lambda'}+\fr{a}_{\Lambda'}+\fr{n}_{\Lambda'}$ which is called the Langlands decomposition. A space $\fr{n}_{\Lambda'}$ is called the nilpotent part of $\fr{q}_{\Lambda'}$. Therefore, in order to understand $\fr{n}_{\Lambda'}$, it suffices to study a subset $\Lambda'$ of a simple root $\Lambda$. Moreover, the simply-connected Lie groups corresponding to the nilpotent Lie algebras obtained by these two methods admit left-invariant Ricci soliton metrics (\cite{T2}, \cite{MT}). 
	
	In this paper, we focus on the case where nilpotent Lie algebras, consisting of the two method described above, are of three-step. Semisimple Lie algebras are classified by Dinkin diagrams. Based on this classification, we determine the dimensions by level for the case where the nilpotent parts of parabolic subalgebras are of three-step (see Section\ref{parabolic}). We list the three-step nilpotent Lie algebras of dimension $\leq 10$ obtained as the nilpotent parts of parabolic subalgebras. Moreover, we list the three-step nilpotent Lie algebras obtained by quivers which are dimension $\leq 10$ up to isomorphisms (see Section\ref{quiver}). Our list prove that there exist nilpotent Lie algebras obtained by quivers which cannot be obtained as nilpotent parts of parabolic subalgebras. Also, the opposite exists. 
	
\section{Parabolic subalgebras}
	In this section, we recall properties of parabolic subalgebras of semisimple Lie algebras. The conjugacy classes of parabolic subalgebras of a semisimple Lie algebra are in one-to-one correspondence with the proper subsets of the simple root system. We mention graded Lie algebras, parabolic subalgebras, and restricted root systems.  
	\begin{Def}
		Let $\fr{g}$ be a semisimple Lie algebra. A direct sum decomposition of $\fr{g}$ as vector spaces into subspaces, 
		\[
			\fr{g}=\bigoplus_{k\in\Z}\fr{g}^k
		\]
		is called a {\it gradation} if $[\fr{g}^p, \fr{g}^q]\subset\fr{g}^{p+q}$ holds for every $p$, $q\in\Z$. 
	\end{Def}
	\begin{Prop}
		Let $\fr{g}$ be a Lie algebra and $D\in\Der(\fr{g})$ be a diagonalizable derivation whose eigenvalues are all integer. Then $D$ provides the gradation that is a graded Lie algebra $\fr{g}=\bigoplus_{k\in\Z}\fr{g}^k$ is given by
		\[
			\fr{g}^k=\{X\in\fr{g}\mid D(X)=kX\}. 
		\]
	\end{Prop}
	\begin{Rem}
		An element $Z\in\fr{g}$ is called a {\it characteristic element} if $\ad_Z\in\Der(\fr{g})$ gives a gradation. 
	\end{Rem}
	\begin{Def}
		Let $\fr{g}$ be a semisimple Lie algebra, $\fr{g}=\fr{k}+\fr{p}$ be the Cartan decomposition, and $\fr{a}$ be a maximal abelian subalgebra of $\fr{p}$. For each $\alpha\in\fr{a}^*$, define a {\it (restricted) root space} $\fr{g}_{\alpha}$ by 
		\[
			\fr{g}_{\alpha}=\{X\in\fr{g}\mid [A, X]=\alpha(A)X\quad \forall A\in\fr{a}\}. 
		\]
		Moreover, define a {\it restricted root system} $\Delta=\Delta(\fr{g}, \fr{a})$ of $\fr{g}$ with respect to $\fr{a}$ by
		\[
			\Delta=\Delta(\fr{g}, \fr{a})=\{\alpha\in\fr{a}^*\mid \alpha\neq0, \fr{g}_{\alpha}\neq0\}. 
		\]
		The following decomposition is called a {\it root space decomposition}, 
		\[
			\fr{g}=\fr{g}_0+\bigoplus_{\alpha\in\Delta}\fr{g}_{\alpha}, 
		\]
		where $\fr{g}_0$ is the centralizer of $\fr{a}$ in $\fr{g}$. 
	\end{Def}
	\begin{Def}
		Let $\fr{g}$ be a semisimple Lie algebra, $\fr{g}=\fr{k}+\fr{p}$ be the Cartan decomposition, $\fr{a}$ be a maximal abelian subalgebra of $\fr{p}$, and $\Delta$ be a restricted root system of $\fr{g}$ with respect to $\fr{a}$. We say a set $\Lambda=\{\alpha_1, \ldots, \alpha_r\}\subset\Delta$ is a {\it simple root system} if the following conditions are satisfied: 
		\begin{enumerate}
			\item $\Lambda$ is basis of $\fr{a}^*$, 
			\item For any $\alpha\in\Lambda$, there exists $c_1, \ldots, c_r\in\Z_{\geq0}$ or $c_1, \ldots, c_r\in\Z_{\leq0}$ such that $\alpha=c_1\alpha_1+\cdots c_r\alpha_r$. 
		\end{enumerate}
		For a root $\alpha=c_1\alpha_1+\cdots c_r\alpha_r\in\Delta$, it is called a {\it positive root} when $c_1, \ldots, c_r\in\Z_{\geq0}$ and a {\it negative root} when $c_1, \ldots, c_r\in\Z_{\leq0}$. Also, a root $\alpha=c_1\alpha_1+\cdots c_r\alpha_r\in\Delta$ is called a {\it highest root} if for any $\alpha=c_1'\alpha_1+\cdots c_r'\alpha_r$, it satisfies $c_1\geq c_1'\ldots c_r\geq c_r'$. 
	\end{Def}
	\begin{Def}
		Let $\fr{g}$ be a semisimple Lie algebra. A subalgebra $\fr{q}$ of $\fr{g}$ is called a {\it parabolic subalgebra} if there exists a nontrivial gradation $\fr{g}=\bigoplus_{k\in\Z}\fr{g}^k$ such that $\fr{q}=\bigoplus_{k\geq0}\fr{g}^k$. 
	\end{Def}
	\begin{Thm}
		Let $\fr{g}$ be a semisimple Lie algebra, $\fr{g}=\fr{k}+\fr{p}$ be the Cartan decomposition, $\fr{a}$ be a maximal abelian subalgebra of $\fr{p}$, $\Delta$ be a restricted root system of $\fr{g}$ with respect to $\fr{a}$, and $\Lambda$ be a simple root system. Then the following hold.  
		\begin{enumerate}
			\item Let $\Lambda'$ be a proper subset of $\Lambda$ and $\langle\Lambda'\rangle$ be a space spanned by $\Lambda'$. Define 
			\[
				\fr{q}_{\Lambda'}=\fr{g}_0+\bigoplus_{\beta\in\Delta^+\cup \langle\Lambda'\rangle}\fr{g}_{\beta}, 
			\]
			where $\Delta^+$ is a set of all positive roots. Then, $\fr{q}_{\Lambda'}$ is a parabolic subalgebra of $\fr{g}$. 
			\item For any parabolic subalgebra $\fr{q}$ of $\fr{g}$, there exists a proper subset $\Lambda'$ of $\Lambda$ such that $\fr{q}$ is isomorphic to $\fr{q}_{\Lambda'}$ up to conjugation. 
		\end{enumerate}
	\end{Thm}
	\begin{Prop}
		Let $\fr{g}=\bigoplus_{k\in\Z}\fr{g}^k$ be a graded Lie algebra and $\fr{q}=\bigoplus_{k\geq0}\fr{g}^k$ be a parabolic subalgebra of $\fr{g}$. Then, the following hold. 
		\begin{enumerate}
			\item $\fr{l}_{\fr{q}}\coloneqq\fr{g}^0$ is a subalgebra of $\fr{q}$, 
			\item $\fr{n}_{\fr{q}}\coloneqq\bigoplus_{k>0}\fr{g}^k$ is a nilpotent subalgebra of $\fr{q}$, 
			\item $\fr{q}=\fr{l}_{\fr{q}}+\fr{n}_{\fr{q}}$ is the direct sum as vector spaces which is called the Chevalley decomposition. 
		\end{enumerate}
	\end{Prop}
	\begin{Rem}
		Note that denote the Chevalley decomposition of a parabolic subalgebra $\fr{q}_{\Lambda'}$ defined by $\Lambda'\subsetneq\Lambda$ as $\fr{l}_{\Lambda'}+\fr{n}_{\Lambda'}$. Every gradation can constructed from $Z\in\fr{a}$. Denote by $\Lambda=\{\alpha_1, \ldots, \alpha_r\}$, and let $\{H^1, \ldots, H^r\}$ be the dual basis of $\Lambda$, that is, $\alpha_i(H^j)=\delta_{ij}$. 
	\end{Rem}
	
	\begin{Prop}
		Let $Z=H^{i_1}+\cdots+H^{i_m}$ satisfying $i_1<\cdots<i_m$ be a characteristic element, $\fr{q}$ be a parabolic subalgebra defined by $Z$, and $\fr{q}=\fr{l}_{\fr{q}}+\fr{n}_{\fr{q}}$ be the Chevalley decomposition. Then the following hold.   
		\begin{enumerate}
			\item $\fr{a}_{\fr{q}}\coloneqq\mathrm{span}\{H^{i_1}, \ldots, H^{i_m}\}$ is an abelian subalgebra of $\fr{l}_{\fr{q}}$, 
			\item $\fr{m}_{\fr{q}}\coloneqq\fr{l}_{\fr{q}}\ominus\fr{a}_{\fr{q}}$ is subalgebra of $\fr{l}_{\fr{q}}$, 
			\item $\fr{q}=\fr{m}_{\fr{q}}+\fr{a}_{\fr{q}}+\fr{n}_{\fr{q}}$ is the direct sum as vector spaces which is called the {\it Langlands decomposition}. 
		\end{enumerate} 
	\end{Prop}
	We call $\fr{n}_{\fr{q}}$ and $\fr{a}_{\fr{q}}+\fr{n}_{\fr{q}}$ a {\it nilpotent part} and a {\it solvable part} of  a parabolic subalgebra $\fr{q}$, responsively.  
	\begin{Prop}\label{nil}
		Let $Z=H^{i_1}+\cdots+H^{i_m}$ satisfying $i_1<\cdots<i_m$ be a characteristic element, $\fr{q}$ be a parabolic subalgebra defined by $Z$, and $\fr{n}_{\fr{q}}$ be a nilpotent part of $\fr{q}$. Then, the following hold. 
		\begin{enumerate}
			\item The number of step of $\fr{n}_{\fr{q}}$ is $\tilde{\alpha}(Z)$, 
			\item If $\fr{n}_{\fr{q}}=\fr{g}^1+\cdots\fr{g}^{r}$, then $\fr{g}^k=\bigoplus_{\alpha(Z)=k}\fr{g}_{\alpha}$. 
		\end{enumerate}
	\end{Prop}
	Therefore, in order to understand $\fr{n}_{\Lambda'}$, it suffices to study a subset $\Lambda'$ of a simple root $\Lambda$.  
\section{Three-step nilpotent parts}\label{parabolic}
	Semisimple Lie algebras are classified by Dinkin diagrams. We refer to \cite{K} for information on positive roots and the highest root of each type. Also, from Proposition\ref{nil}, a characteristic element $Z$ for which the nilpotent part is of three-step is determined by the highest root. Regarding multiplicities, we refer to \cite{T1}.
	\subsection{Nilpotent parts of parabolic subalgebras of $A_n$}
		A set of all positive roots, the highest root of $A_n$, characteristic elements where nilpotent parts are of three-step are as follow: 
		\begin{itemize}
			\item a set of all positive roots $\Delta^+=\{e_i-e_j\mid 1\leq i< j\leq n+1\}$, 
			\item $\alpha_i=e_i-e_{i+1}$ for $1\leq i\leq n$, 
			\item the highest root $\tilde{\alpha}=\alpha_1+\cdots+\alpha_n$, 
			\item $Z=H^{k_1}+H^{k_2}+H^{k_3}$ with $1\leq k_1<k_2<k_3\leq n$. 
		\end{itemize}
		If we count the positive roots of $Z$ by level, the results are shown as in Table\ref{A_n}. 
		\begin{table}[H]
			\begin{center}
				\begin{tabular}{|c|c|}
					\hline
					level & the number of roots $e_i-e_j$   \\ \hline
					3 & $k_1(n-k_3+1)$  \\
					2 & $k_1(k_3-k_2)+(k_2-k_1)(n-k_3+1)$ \\
					1 & $(k_2-k_1)(k_3-k_2+k_1)+(k_3-k_2)(n-k_3+1)$  \\ \hline
				\end{tabular}
			\end{center}
			\caption{The number of positive roots by $H^{k_1}+H^{k_2}+H^{k_3}$ on $A_n$, by level}
			\label{A_n}
		\end{table}
		Let $\ell_i$ be the number of level $i$ roots for $i=1, 2, 3$. When all cases are calculated, the result is as shown in Table\ref{A_n-multi}. 
		\begin{table}[H]
            		\begin{center}
            			\begin{tabular}{|c|c|c|c|}
            				\hline
            				multiplicities & $\mathrm{dim}~\fr{g}^3$ & $\mathrm{dim}~\fr{g}^2$ & $\mathrm{dim}~\fr{g}^1$ \\ \hline
            				1 & $\ell_3$ & $\ell_2$ & $\ell_1$ \\ \hline
            				2 & $2\ell_3$ & $2\ell_2$ & $2\ell_1$ \\ \hline	
            				4 & $4\ell_3$ & $4\ell_2$ & $4\ell_1$ \\ \hline			
            			\end{tabular}
            		\end{center}
			\caption{Dimensions of the nilpotent parts by $H^{k_1}+H^{k_2}+H^{k_3}$, by level}
			\label{A_n-multi}
		\end{table}
	\subsection{Nilpotent parts of parabolic subalgebras of $B_n$}
		A set of all positive roots, the highest root of $B_n$, characteristic elements where nilpotent parts are of three-step are as follow: 
		\begin{itemize}
			\item a set of all positive roots $\Delta^+=\{e_i\pm e_j\mid 1\leq i <j\leq n\}\cup\{e_i \mid1\leq i\leq n\}$, 
			\item $\alpha_i=e_i-e_{i+1}$ for $1\leq i\leq n-1$, $\alpha_n=e_n$
			\item the highest root $\tilde{\alpha}=\alpha_1+2\alpha_2+\cdots+2\alpha_n$,  
			\item $Z=H^1+H^k$ with $2\leq k\leq n$. 
		\end{itemize}
		If we count the positive roots of $Z$ by level, the results are shown as in Table\ref{B_n}. 
		\begin{table}[H]
            		\begin{center}
            			\begin{tabular}{|c|c|c|}
            				\hline
            				level & the number of long roots $e_i\pm e_j$ & the number of short roots $e_i$ \\ \hline
            				3 & $k-1$ & 0 \\
            				2 & $k(n-k+1)-2$ & 1 \\
            				1 & $(k-1)(2(n-k)+1)$ & $k-1$ \\ \hline
            			\end{tabular}
            		\end{center}
			\caption{The number of positive roots by $H^1+H^k$ on $B_n$, by level}
			\label{B_n}
		\end{table}
		When all cases are calculated, the result is as shown in Table\ref{B_n-multi}. 
		\begin{table}[H]
            		\begin{center}
            			\begin{tabular}{|c|c|c|c|c|}
            				\hline
            				\multicolumn{2}{|c|}{multiplicities} & \multirow{2}{*}{$\mathrm{dim}~\fr{g}^3$} & \multirow{2}{*}{$\mathrm{dim}~\fr{g}^2$} & \multirow{2}{*}{$\mathrm{dim}~\fr{g}^1$} \\
            				\cline{1-2}
            				long & short &  &  &  \\ \hline
            				2 & 2 & $2(k-1)$ & $2(k(n-k+1)-1)$ & $4(k-1)(n-k+1)$ \\ \hline
            				1 & $m$ & $k-1$ & $k(n-k+1)+m-2$ & $(k-1)(2(n-k)+m+1)$  \\ \hline
            			\end{tabular}
            		\end{center}
			\caption{Dimensions of the nilpotent parts by $H^1+H^k$, by level}
			\label{B_n-multi}
		\end{table}
	\subsection{Nilpotent parts of parabolic subalgebras of $C_n$}
		A set of all positive roots, the highest root of $C_n$, characteristic elements where nilpotent parts are of three-step are as follow: 
		\begin{itemize}
			\item a set of all positive roots $\Delta^+=\{e_i\pm e_j \mid 1\leq i<j\leq n\}\cup\{2e_i\mid 1\leq i\leq n\}$, 
			\item $\alpha_i=e_i-e_{i+1}$ for $1\leq i\leq n-1$, $\alpha_n=2e_n$
			\item the highest root $\tilde{\alpha}=2\alpha_1+\cdots+2\alpha_{n-1}+\alpha_n$, 
			\item $Z=H^k+H^n$ with $1\leq k\leq n-1$. 
		\end{itemize}
		If we count the positive roots of $Z$ by level, the results are shown as in Table\ref{C_n}.  
		\begin{table}[H]
            		\begin{center}
            			\begin{tabular}{|c|c|c|}
            				\hline
            				level & the number of short roots $e_i\pm e_j$ & the number of long roots $(2e_i)$ \\ \hline
            				3 & $(1/2)k(k-1)$ & $k$ \\
            				2 & $k(n-k)$ & 0 \\
            				1 & $(1/2)(n-k)(n+k-1)$ & $n-k$ \\ \hline
            			\end{tabular}
            		\end{center}
			\caption{The number of positive roots by $H^k+H^n$ on $C_n$, by level}
			\label{C_n}
		\end{table}
		When all cases are calculated, the result is as shown in Table\ref{C_n-multi}. 
		\begin{table}[H]
            		\begin{center}
            			\begin{tabular}{|c|c|c|c|c|}
            				\hline
            				\multicolumn{2}{|c|}{multiplicities} & \multirow{2}{*}{$\mathrm{dim}~\fr{g}^3$} & \multirow{2}{*}{$\mathrm{dim}~\fr{g}^2$} & \multirow{2}{*}{$\mathrm{dim}~\fr{g}^1$} \\
            				\cline{1-2}
            				short & long &  &  &  \\ \hline
            				1 & 1 & $(1/2)k(k+1)$ & $k(n-k)$ & $(1/2)(n-k)(n+k+1)$ \\ \hline
            				2 & 2 & $k(k+1)$ & $2k(n-k)$ & $(n-k)(n+k+1)$  \\ \hline
            				4 & 3 & $k(2k+1)$ & $4k(n-k)$ & $(n-k)(2(n+k)+1)$  \\ \hline
            				2 & 1 & $k^2$ & $2k(n-k)$ & $(n-k)(n+k)$  \\ \hline
            				4 & 1 & $k(2k-1)$ & $4k(n-k)$ & $(n-k)(2(n+k)-1)$  \\ \hline
            			\end{tabular}
            		\end{center}
			\caption{Dimensions of the nilpotent parts by $H^k+H^n$, by level}
			\label{C_n-multi}
		\end{table}
	\subsection{Nilpotent parts of parabolic subalgebras of $D_n$}
		A set of all positive roots, the highest root of $D_n$, characteristic elements where nilpotent parts are of three-step are as follow: 
		\begin{itemize}
			\item a set of all positive roots $\Delta^+=\{e_i\pm e_j \mid 1\leq i<j\leq n\}$, 
			\item $\alpha_i=e_i-e_{i+1}$ for $1\leq i\leq n-1$, $\alpha_n=e_{n-1}+e_n$, 
			\item the highest root $\tilde{\alpha}=\alpha_1+2\alpha_2+\cdots+2\alpha_{n-2}+\alpha_{n-1}+\alpha_n$, 
			\item $Z=H^1+H^{n-1}+H^n, H^k+H^{\ell}$ with $2\leq k\leq n-2$, $\ell=1, n-1, n$. 
		\end{itemize}
		Since the Dinkin diagram is symmetric is cases $Z=H^k+H^{n-1}$ and $Z=H^k+H^n$, we can examine them together. Therefore, we count the positive roots of $Z$ by level in three cases--$Z=H^1+H^{n-1}+H^n$, $Z=H^1+H^k$ and $Z=H^k+H^{n-1}(H^k+H^n)$. First, In the case of $Z=H^1+H^{n-1}+H^n$. the results are shown as in Table\ref{D_n-1}. 
		\begin{table}[H]
            		\begin{center}
            			\begin{tabular}{|c|c|c|}
            				\hline
            				level & the number of roots $e_i\pm e_j$  \\ \hline
            				3 & $n-2$  \\
            				2 & $(1/2)(n-3)(n-2)+2$ \\
            				1 & $3(n-2)$  \\ \hline
            			\end{tabular}
            		\end{center}
			\caption{The number of positive roots by $H^1+H^{n-1}+H^n$ on $D_n$, by level}
			\label{D_n-1}
		\end{table}
		When all cases are calculated, the result is as shown in Table\ref{D_n-multi-1}. 
		\begin{table}[H]
            		\begin{center}
            			\begin{tabular}{|c|c|c|c|}
            				\hline
            				multiplication& $\mathrm{dim}~\fr{g}^3$ & $\mathrm{dim}~\fr{g}^2$ & $\mathrm{dim}~\fr{g}^1$ \\ \hline
            				1 & $n-2$ & $(1/2)(n-3)(n-2)+2$ & $3(n-2)$ \\ \hline
            				2 & $2(n-2)$ & $(n-3)(n-2)+4$ & $6(n-2)$ \\ \hline				
            			\end{tabular}
            		\end{center}
			\caption{Dimensions of the nilpotent parts by $H^1+H^{n-1}+H^n$, by level}
			\label{D_n-multi-1}
		\end{table}
		Next, in the case of $Z=H^1+H^k$. the results are shown as in Table\ref{D_n-2}. 
		\begin{table}[H]
            		\begin{center}
            			\begin{tabular}{|c|c|c|}
            				\hline
            				level & the number of roots $e_i\pm e_j$  \\ \hline
            				3 & $k-1$  \\
            				2 & $2(n-k)+(1/2)(k-2)(k-1)$ \\
            				1 & $(k-1)(2(n-k)+1)$  \\ \hline
            			\end{tabular}
            		\end{center}
			\caption{The number of positive roots by $H^1+H^k$ on $D_n$, by level}
			\label{D_n-2}
		\end{table}
		When all cases are calculated, the result is as shown in Table\ref{D_n-multi-2}. 
		\begin{table}[H]
            		\begin{center}
            			\begin{tabular}{|c|c|c|c|}
            				\hline
            				multiplication& $\mathrm{dim}~\fr{g}^3$ & $\mathrm{dim}~\fr{g}^2$ & $\mathrm{dim}~\fr{g}^1$ \\ \hline
            				1 & $k-1$ & $2(n-k)+(1/2)(k-2)(k-1)$ & $(k-1)(2(n-k)+1)$ \\ \hline
            				2 & $2(k-1)$ & $4(n-k)+(k-2)(k-1)$ & $2(k-1)(2(n-k)+1)$ \\ \hline				
            			\end{tabular}
            		\end{center}
			\caption{Dimensions of the nilpotent parts by $H^1+H^k$, by level}
			\label{D_n-multi-2}
		\end{table}
		Finally, in the case of $Z=H^k+H^{n-1}$ or $H^k+H^n$. the results are shown as in Table\ref{D_n-3}. 
		\begin{table}[H]
            		\begin{center}
            			\begin{tabular}{|c|c|c|}
            				\hline
            				level & the number of roots $e_i\pm e_j$  \\ \hline
            				3 & $(1/2)k(k-1)$  \\
            				2 & $k(n-k)$ \\
            				1 & $(1/2)(n-k)(n+k-1)$  \\ \hline
            			\end{tabular}
            		\end{center}
			\caption{The number of positive roots by $H^k+H^{n-1}$ or $H^k+H^n$ on $D_n$, by level}
			\label{D_n-3}
		\end{table}
		When all cases are calculated, the result is as shown in Table\ref{D_n-multi-3}. 
		\begin{table}[H]
            		\begin{center}
            			\begin{tabular}{|c|c|c|c|}
            				\hline
            				multiplication& $\mathrm{dim}~\fr{g}^3$ & $\mathrm{dim}~\fr{g}^2$ & $\mathrm{dim}~\fr{g}^1$ \\ \hline
            				1 & $(1/2)k(k-1)$ & $k(n-k)$ & $(1/2)(n-k)(n+k-1)$ \\ \hline
            				2 & $k(k-1)$ & $2k(n-k)$ & $(n-k)(n+k-1)$ \\ \hline				
            			\end{tabular}
            		\end{center}
			\caption{Dimensions of the nilpotent parts by $H^k+H^{n-1}$ or $H^k+H^n$, by level}
			\label{D_n-multi-3}
		\end{table}
	\subsection{Nilpotent parts of parabolic subalgebras of $E_6$}
		A set of all positive roots, the highest root of $E_6$, characteristic elements where nilpotent parts are of three-step are as follow: 
		\begin{itemize}
			\item a set of all positive roots refers to \cite{K}
			\item the highest root $\tilde{\alpha}=\alpha_1+2\alpha_2+2\alpha_3+3\alpha_4+2\alpha_5+\alpha_6$, 
			\item $Z=H^4, H^k+H^{\ell}$ with $k=1, 6$, $\ell=2, 3, 5$. 
		\end{itemize}
		When the Dinkin diagram is symmetrical, we examine them as a group. First,  in the case of $Z=H^4$. the results are shown as in Table\ref{E_6-1}. 
		\begin{table}[H]
            		\begin{center}
            			\begin{tabular}{|c|c|}
            				\hline
            				level & the number of roots $e_i- e_j$ \\ \hline
            				3 & 2  \\
            				2 & 9  \\
            				1 & 18  \\ \hline
            			\end{tabular}
            		\end{center}
			\caption{The number of positive roots by $H^4$ on $E_6$, by level}
			\label{E_6-1}
		\end{table}
		When all cases are calculated, the result is as shown in Table\ref{E_6-multi-1}. 
		\begin{table}[H]
            		\begin{center}
            			\begin{tabular}{|c|c|c|c|}
            				\hline
            				multiplication& $\mathrm{dim}~\fr{g}^3$ & $\mathrm{dim}~\fr{g}^2$ & $\mathrm{dim}~\fr{g}^1$ \\ \hline
            				1 & 2 & 9 & 18 \\ \hline
            				2 & 4 & 18 & 36 \\ \hline				
            			\end{tabular}
            		\end{center}
			\caption{Dimensions of the nilpotent parts by $H^4$, by level}
			\label{E_6-multi-1}
		\end{table}
		Second, in the case of $Z=H^1+H^3$ or $H^5+H^6$. the results are shown as in Table\ref{E_6-2}. 
		\begin{table}[H]
            		\begin{center}
            			\begin{tabular}{|c|c|}
            				\hline
            				level& the number of roots $e_i- e_j$ \\ \hline
            				3 & 5  \\
            				2 & 10  \\
            				1 & 11  \\ \hline
            			\end{tabular}
            		\end{center}
			\caption{The number of positive roots by $H^1+H^3$ or $H^5+H^6$ on $E_6$, by level}
			\label{E_6-2}
		\end{table}
		When all cases are calculated, the result is as shown in Table\ref{E_6-multi-2}. 
		\begin{table}[H]
            		\begin{center}
            			\begin{tabular}{|c|c|c|c|}
            				\hline
            				multiplication& $\mathrm{dim}~\fr{g}^3$ & $\mathrm{dim}~\fr{g}^2$ & $\mathrm{dim}~\fr{g}^1$ \\ \hline
            				1 & 5 & 10 & 11 \\ \hline
            				2 & 10 & 20 & 22 \\ \hline				
            			\end{tabular}
            		\end{center}
			\caption{Dimensions of the nilpotent parts by $H^1+H^3$ or $H^5+H^6$, by level}
			\label{E_6-multi-2}
		\end{table}
		Third, in the case of $Z=H^1+H^2$ or $H^2+H^6$. the results are shown as in Table\ref{E_6-3}. 
		\begin{table}[H]
            		\begin{center}
            			\begin{tabular}{|c|c|}
            				\hline
            				level & the number of roots $e_i- e_j$ \\ \hline
            				3 & 1  \\
            				2 & 10  \\
            				1 & 15  \\ \hline
            			\end{tabular}
            		\end{center}
			\caption{The number of positive roots by $H^1+H^2$ or $H^2+H^6$ on $E_6$, by level}
			\label{E_6-3}
		\end{table}
		When all cases are calculated, the result is as shown in Table\ref{E_6-multi-3}. 
		\begin{table}[H]
            		\begin{center}
            			\begin{tabular}{|c|c|c|c|}
            				\hline
            				multiplication& $\mathrm{dim}~\fr{g}^3$ & $\mathrm{dim}~\fr{g}^2$ & $\mathrm{dim}~\fr{g}^1$ \\ \hline
            				1 & 1 & 10 & 15 \\ \hline
            				2 & 2 & 20 & 30 \\ \hline				
            			\end{tabular}
            		\end{center}
			\caption{Dimensions of the nilpotent parts by $H^1+H^2$ or $H^2+H^6$, by level}
			\label{E_6-multi-3}
		\end{table}
		Finally, in the case of $Z=H^1+H^5$ or $H^3+H^6$. the results are shown as in Table\ref{E_6-4}.
		\begin{table}[H]
            		\begin{center}
            			\begin{tabular}{|c|c|}
            				\hline
            				level & the number of roots $e_i- e_j$ \\ \hline
            				3 & 4  \\
            				2 & 10  \\
            				1 & 16  \\ \hline
            			\end{tabular}
            		\end{center}
			\caption{The number of positive roots by $H^1+H^5$ or $H^3+H^6$, by level}
			\label{E_6-4}
		\end{table}
		When all cases are calculated, the result is as shown in Table\ref{E_6-multi-4}. 
		\begin{table}[H]
            		\begin{center}
            			\begin{tabular}{|c|c|c|c|}
            				\hline
            				multiplication& $\mathrm{dim}~\fr{g}^3$ & $\mathrm{dim}~\fr{g}^2$ & $\mathrm{dim}~\fr{g}^1$ \\ \hline
            				1 & 4 & 10 & 16 \\ \hline
            				2 & 8 & 20 & 32 \\ \hline				
            			\end{tabular}
            		\end{center}
			\caption{Dimensions of the nilpotent parts by $H^1+H^5$ or $H^3+H^6$, by level}
			\label{E_6-multi-4}
		\end{table}
	\subsection{Nilpotent parts of parabolic subalgebras of $E_7$}
		A set of all positive roots, the highest root of $E_7$, characteristic elements where nilpotent parts are of three-step are as follow: 
		\begin{itemize}
			\item a set of all positive roots refers to \cite{K}
			\item the highest root $\tilde{\alpha}=2\alpha_1+2\alpha_2+3\alpha_3+4\alpha_4+3\alpha_5+2\alpha_6+\alpha_7$, 
			\item $Z=H^3, H^5, H^k+H^7$ with $k=1, 2, 6$. 
		\end{itemize}
		First, in the case of $Z=H^3$. the results are shown as in Table\ref{E_7-1}. 
		\begin{table}[H]
            		\begin{center}
            			\begin{tabular}{|c|c|}
            				\hline
            				level & the number of roots $e_i- e_j$ \\ \hline
            				3 & 2  \\
            				2 & 15  \\
            				1 & 30  \\ \hline
            			\end{tabular}
            		\end{center}
			\caption{The number of positive roots by $H^3$, by level}
			\label{E_7-1}
		\end{table}
		When all cases are calculated, the result is as shown in Table\ref{E_7-multi-1}. 
		\begin{table}[H]
            		\begin{center}
            			\begin{tabular}{|c|c|c|c|}
            				\hline
            				multiplication& $\mathrm{dim}~\fr{g}^3$ & $\mathrm{dim}~\fr{g}^2$ & $\mathrm{dim}~\fr{g}^1$ \\ \hline
            				1 & 2 & 15 & 30 \\ \hline
            				2 & 4 & 30 & 60 \\ \hline				
            			\end{tabular}
            		\end{center}
			\caption{Dimensions of the nilpotent parts by $H^3$, by level}
			\label{E_7-multi-1}
		\end{table}
		Second, in the case of $Z=H^5$ the results are shown as in Table\ref{E_7-2}. 
		\begin{table}[H]
            		\begin{center}
            			\begin{tabular}{|c|c|}
            				\hline
            				level & the number of roots $e_i- e_j$ \\ \hline
            				3 & 5  \\
            				2 & 15  \\
            				1 & 30  \\ \hline
            			\end{tabular}
            		\end{center}
			\caption{The number of positive roots by $H^5$, by level}
			\label{E_7-2}
		\end{table}
		When all cases are calculated, the result is as shown in Table\ref{E_7-multi-2}. 
		\begin{table}[H]
            		\begin{center}
            			\begin{tabular}{|c|c|c|c|}
            				\hline
            				multiplication& $\mathrm{dim}~\fr{g}^3$ & $\mathrm{dim}~\fr{g}^2$ & $\mathrm{dim}~\fr{g}^1$ \\ \hline
            				1 & 5 & 15 & 30 \\ \hline
            				2 & 10 & 30 & 60 \\ \hline				
            			\end{tabular}
            		\end{center}
			\caption{Dimensions of the nilpotent parts by $H^5$, by level}
			\label{E_7-multi-2}
		\end{table}
		Third, in the case of $Z=H^1+H^7$ the results are shown as in Table\ref{E_7-3}. 
		\begin{table}[H]
            		\begin{center}
            			\begin{tabular}{|c|c|}
            				\hline
            				level & the number of roots $e_i- e_j$ \\ \hline
            				3 & 1 \\
            				2 & 16  \\
            				1& 26  \\ \hline
            			\end{tabular}
            		\end{center}
			\caption{The number of positive roots by $H^1+H^7$, by level}
			\label{E_7-3}
		\end{table}
		When all cases are calculated, the result is as shown in Table\ref{E_7-multi-3}. 
		\begin{table}[H]
            		\begin{center}
            			\begin{tabular}{|c|c|c|c|}
            				\hline
            				multiplication& $\mathrm{dim}~\fr{g}^3$ & $\mathrm{dim}~\fr{g}^2$ & $\mathrm{dim}~\fr{g}^1$ \\ \hline
            				1 & 1& 16 & 26 \\ \hline
            				2 & 12 & 32 & 52 \\ \hline				
            			\end{tabular}
            		\end{center}
			\caption{Dimensions of the nilpotent parts by $H^1+H^7$, by level}
			\label{E_7-multi-3}
		\end{table}
		Fourth, in the case of $Z=H^2+H^7$ the results are shown as in Table\ref{E_7-4}. 
		\begin{table}[H]
            		\begin{center}
            			\begin{tabular}{|c|c|}
            				\hline
            				level & the number of roots $e_i- e_j$ \\ \hline
            				3 & 6  \\
            				2 & 16  \\
            				1 & 26  \\ \hline
            			\end{tabular}
            		\end{center}
			\caption{The number of positive roots by $H^2+H^7$, by level}
			\label{E_7-4}
		\end{table}
		When all cases are calculated, the result is as shown in Table\ref{E_7-multi-4}. 
		\begin{table}[H]
            		\begin{center}
            			\begin{tabular}{|c|c|c|c|}
            				\hline
            				multiplication& $\mathrm{dim}~\fr{g}^3$ & $\mathrm{dim}~\fr{g}^2$ & $\mathrm{dim}~\fr{g}^1$ \\ \hline
            				1 & 6 & 16 & 26 \\ \hline
            				2 & 12 & 32 & 52 \\ \hline				
            			\end{tabular}
            		\end{center}
			\caption{Dimensions of the nilpotent parts by $H^2+H^7$, by level}
			\label{E_7-multi-4}
		\end{table}
		Finally, in the case of $Z=H^6+H^7$ the results are shown as in Table\ref{E_7-5}. 
		\begin{table}[H]
            		\begin{center}
            			\begin{tabular}{|c|c|}
            				\hline
            				level & the number of roots $e_i- e_j$ \\ \hline
            				3 & 10  \\
            				2 & 16  \\
            				1 & 17  \\ \hline
            			\end{tabular}
            		\end{center}
			\caption{The number of positive roots by $H^6+H^7$, by level}
			\label{E_7-5}
		\end{table}
		When all cases are calculated, the result is as shown in Table\ref{E_7-multi-5}.
		 \begin{table}[H]
            		\begin{center}
            			\begin{tabular}{|c|c|c|c|}
            				\hline
            				multiplication& $\mathrm{dim}~\fr{g}^3$ & $\mathrm{dim}~\fr{g}^2$ & $\mathrm{dim}~\fr{g}^1$ \\ \hline
            				1 & 10 & 16 & 17 \\ \hline
            				2 & 20 & 32 & 34 \\ \hline				
            			\end{tabular}
            		\end{center}
			\caption{Dimensions of the nilpotent parts by $H^6+H^7$, by level}
			\label{E_7-multi-5}
		\end{table}
	\subsection{Nilpotent parts of parabolic subalgebras of $E_8$}
		A set of all positive roots, the highest root of $E_8$, characteristic elements where nilpotent parts are of three-step are as follow: 
		\begin{itemize}
			\item a set of all positive roots refers to \cite{K}
			\item the highest root $\tilde{\alpha}=2\alpha_1+3\alpha_2+4\alpha_3+6\alpha_4+5\alpha_5+4\alpha_6+3\alpha_7+2\alpha_8$. 
			\item $Z=H^2, H^7$. 
		\end{itemize}
		First, in the case of $Z=H^2$. the results are shown as in Table\ref{E_8-1}.
		\begin{table}[H]
            		\begin{center}
            			\begin{tabular}{|c|c|}
            				\hline
            				level & the number of roots $e_i- e_j$ \\ \hline
            				3 & 8  \\
            				2 & 29  \\
            				1 & 56  \\ \hline
            			\end{tabular}
            		\end{center}
			\caption{The number of positive roots by $H^2$, by level}
			\label{E_8-1}
		\end{table}
		When all cases are calculated, the result is as shown in Table\ref{E_8-multi-1}.
		\begin{table}[H]
            		\begin{center}
            			\begin{tabular}{|c|c|c|c|}
            				\hline
            				multiplication& $\mathrm{dim}~\fr{g}^3$ & $\mathrm{dim}~\fr{g}^2$ & $\mathrm{dim}~\fr{g}^1$ \\ \hline
            				1 & 8 & 29 & 56 \\ \hline
            				2 & 16 & 58 & 112 \\ \hline				
            			\end{tabular}
            		\end{center}
			\caption{Dimensions of the nilpotent parts by $H^2$, by level}
			\label{E_8-multi-1}
		\end{table}
		Next, in the case of $Z=H^7$. the results are shown as in Table\ref{E_8-2}.
		\begin{table}[H]
            		\begin{center}
            			\begin{tabular}{|c|c|}
            				\hline
            				level & the number of roots $e_i- e_j$ \\ \hline
            				3 & 2  \\
            				2 & 26  \\
            				1 & 55  \\ \hline
            			\end{tabular}
            		\end{center}
			\caption{The number of positive roots by $H^7$, by level}
			\label{E_8-2}
		\end{table}
		When all cases are calculated, the result is as shown in Table\ref{E_8-multi-2}.
		\begin{table}[H]
            		\begin{center}
            			\begin{tabular}{|c|c|c|c|}
            				\hline
            				multiplication& $\mathrm{dim}~\fr{g}^3$ & $\mathrm{dim}~\fr{g}^2$ & $\mathrm{dim}~\fr{g}^1$ \\ \hline
            				1 & 2 & 26 & 55 \\ \hline
            				2 & 4 & 52 & 110 \\ \hline				
            			\end{tabular}
            		\end{center}
			\caption{Dimensions of the nilpotent parts by $H^7$, by level}
			\label{E_8-multi-2}
		\end{table}
	\subsection{Nilpotent parts of parabolic subalgebras of $F_4$}
		A set of all positive roots, the highest root of $F_4$, characteristic elements where nilpotent parts are of three-step are as follow: 
		\begin{itemize}
			\item a set of all positive roots refers to \cite{K}
			\item the highest root $\tilde{\alpha}=2\alpha_1+4\alpha_2+3\alpha_3+2\alpha_4$, 
			\item $Z=H^3$. 
		\end{itemize}
		If we count the positive roots of $Z$ by level, the results are shown as in Table\ref{F_4}. 
		\begin{table}[H]
            		\begin{center}
            			\begin{tabular}{|c|c|c|c|}
            				\hline
            				level & the number of long roots & the number of short roots \\ \hline
            				3 & 2 & 0  \\
            				2 & 2 & 3  \\
            				1 & 5 & 6  \\ \hline
            			\end{tabular}
            		\end{center}
			\caption{The number of positive roots by $H^3$, by level}
			\label{F_4}
		\end{table}
		When all cases are calculated, the result is as shown in Table\ref{F_4-multi}.
		\begin{table}[H]
            		\begin{center}
            			\begin{tabular}{|c|c|c|c|c|}
            				\hline
            				\multicolumn{2}{|c|}{multiplication} & \multirow{2}{*}{$\mathrm{dim}~\fr{g}^3$} & \multirow{2}{*}{$\mathrm{dim}~\fr{g}^2$} & \multirow{2}{*}{$\mathrm{dim}~\fr{g}^1$} \\
            				\cline{1-2}
            				long & short &  &  &  \\ \hline
            				1 & 1 & 2 & 5 & 11 \\ \hline
            				2 & 2 & 4 & 10 & 22  \\ \hline
            				1 & 2 & 2 & 8 & 17 \\ \hline
            				1 & 4 & 2 & 14 & 29 \\ \hline
            				1 & 8 & 2 & 26 & 53 \\ \hline
            			\end{tabular}
            		\end{center}
			\caption{Dimensions of the nilpotent parts by $H^3$, by level}
			\label{F_4-multi}
		\end{table}
	\subsection{Nilpotent parts of parabolic subalgebras of $G_2$}
		A set of all positive roots, the highest root of $G_2$, characteristic elements where nilpotent parts are of three-step are as follow: 
		\begin{itemize}
			\item a set of all positive roots refers to \cite{K}
			\item the highest root $\tilde{\alpha}=3\alpha_1+2\alpha_2$, 
			\item $Z=H^1$. 
		\end{itemize}
		If we count the positive roots of $Z$ by level, the results are shown as in Table\ref{G_2}. 
		\begin{table}[H]
            		\begin{center}
            			\begin{tabular}{|c|c|c|}
            				\hline
            				level & the number of short roots & the number of long roots \\ \hline
            				3 & 0 & 2  \\
            				2 & 1 & 0 \\
            				1 & 2 & 1\\ \hline
            			\end{tabular}
            		\end{center}
			\caption{The number of positive roots by $H^1$, by level}
			\label{G_2}
		\end{table}
		When all cases are calculated, the result is as shown in Table\ref{G_2-multi}.
		\begin{table}[H]
            		\begin{center}
            			\begin{tabular}{|c|c|c|c|c|}
            				\hline
            				\multicolumn{2}{|c|}{multiplication} & \multirow{2}{*}{$\mathrm{dim}~\fr{g}^3$} & \multirow{2}{*}{$\mathrm{dim}~\fr{g}^2$} & \multirow{2}{*}{$\mathrm{dim}~\fr{g}^1$} \\
            				\cline{1-2}
            				long & short &  &  &  \\ \hline
            				1 & 1 & 2 & 1 & 2 \\ \hline
            				2 & 2 & 4 & 2 & 4  \\ \hline
            			\end{tabular}
            		\end{center}
			\caption{Dimensions of the nilpotent parts by $H^1$, by level}
			\label{G_2-multi}
		\end{table}
	\subsection{Nilpotent parts of parabolic subalgebras of $BC_n$}
		A set of all positive roots and the highest root of $BC_n$ are as follow: 
		\begin{itemize}
			\item a set of all positive roots $\Delta^+=\{e_i\pm e_j\mid 1\leq i <j\leq n\}\cup\{e_i \mid1\leq i\leq n\}\cup\{2e_i\mid 1\leq i\leq n\}$, 
			\item $\alpha_i=e_i-e_{i+1}$ for $1\leq i\leq n-1$, $\alpha_n=e_n$
			\item the highest root $\tilde{\alpha}=2\alpha_1+\cdots+2\alpha_n$. 
		\end{itemize}
		Therefore, there do not exist characteristic elements $Z$ where the nilpotent part by $Z$ is of three-step. 
	\subsection{low-dimensional nilpotent parts}
	\begin{Thm}
		Nilpotent Lie algebras as obtained nilpotent parts which is dimension $\leq 10$ of parabolic subalgebras are as shown in Table \ref{parabolic table}: 
	\end{Thm}
	\begin{table}[H]
		\begin{center}
			\begin{tabular}{|c|c|c|c|c|c|c|}
				\hline
				$\mathrm{dim}~\fr{g}$ & type & multiplicities & $Z$ & $\mathrm{dim}~\fr{g}^3$ & $\mathrm{dim}~\fr{g}^2$ & $\mathrm{dim}~\fr{g}^1$ \\ \hline
				4 & $B_2(C_2)$ & $(1, 1)$ & $H^1+H^2$ & 1 & 1 & 2 \\ \hline
				5 & $G_2$ & $(1, 1)$ & $H^1$ & 2 & 1 & 2 \\ \hline
				\multirow{2}{*}{6} & $A_3(D_3)$ & $(1, 1, 1)$ & $H^1+H^2+H^3$ & 1 & 2 & 3 \\ 
				\cline{2-7}
				 & $B_2(C_2)$ & $(1, 2)((2, 1))$ & $H^1+H^2$ & 1 & 2 & 3 \\ \hline
				\multirow{6}{*}{8} & $B_2(C_2)$ & $(2, 2)$ & $H^1+H^2$ & 2 & 2 & 4 \\
				\cline{2-7}
				 & $B_2$ & $(1, 3)$ & $H^1+H^2$ & 1 & 3 & 4 \\
				\cline{2-7}
				 & $B_3$ & $(1, 1, 1)$ & $H^1+H^2$ & 1 & 3 & 4 \\
				\cline{2-7}
				 & $C_3$ & $(1, 1, 1)$ & $H^1+H^3$ & 1 & 2 & 5 \\
				\cline{2-7}
				 & $C_3$ & $(1, 1, 1)$ & $H^2+H^3$ & 3 & 2 & 3 \\ 
				\cline{2-7}
				 & $B_3$ & $(1, 1, 1)$ & $H^1+H^3$ & 2 & 2 & 4 \\ \hline
				\multirow{4}{*}{9} & \multirow{2}{*}{$A_4$} & \multirow{2}{*}{$(1, 1, 1, 1)$} & $H^1+H^2+H^3$ & \multirow{2}{*}{2} & \multirow{2}{*}{3} & \multirow{2}{*}{4} \\
				 & & & ($H^2+H^3+H^4$) &  &  &  \\
				\cline{2-7}
				 & \multirow{2}{*}{$A_4$} &\multirow{2}{*}{$(1, 1, 1, 1)$} & $H^1+H^2+H^4$ & \multirow{2}{*}{1} & \multirow{2}{*}{3} & \multirow{2}{*}{5} \\
				 & & & ($H^1+H^3+H^4$) & & & \\ \hline
				\multirow{5}{*}{10} & $B_2(C_2)$ & $(1, 4) ((4, 1))$ & $H^1+H^2$ & 1 & 4 & 5 \\
				\cline{2-7}
				 & $B_3$ & $(1, 1, 2)$ & $H^1+H^2$ & 1 & 4 & 5 \\
				\cline{2-7}
				 & \multirow{2}{*}{$D_4$} & \multirow{2}{*}{$(1, 1, 1, 1)$} & $H^1+H^2$ ($H^2+H^3$) & \multirow{2}{*}{1} & \multirow{2}{*}{4} & \multirow{2}{*}{5} \\
				 &  & & $H^2+H^3$, $H^2+H^4$ & & &  \\
				\cline{2-7}
				 & $G_2$ & $(2, 2)$ & $H^1$ & 4 & 2 & 4 \\ \hline
 			\end{tabular}
		\end{center}
		\caption{Nilpotent parts parabolic subalgebras which are $\leq\mathrm{dim} 10$}
		 \label{parabolic table}
	\end{table}
	\begin{proof}
		We carry out the computation for each root system separately. Since the arguments are the same in all cases, we present only the case of type $C_n$. Consider the three-step nilpotent parts $\fr{g}$ of parabolic subalgebras of $C_n$-type semisimple Lie algebras with multiplicities $(1, 1, \ldots, 1)$. From Table\ref{C_n-multi}, it satisfies the following: 
		\begin{itemize}
			\item $\mathrm{dim}~\fr{g}^3=(1/2)k(k+1)$, 
			\item $\mathrm{dim}~\fr{g}^2=k(n-k)$, 
			\item $\mathrm{dim}~\fr{g}^1=(1/2)(n-k)(n+k+1)$
		\end{itemize}
		where $1\leq k\leq n-1$. Consider $\mathrm{dim}~\fr{g}\leq10$. 
		\begin{align*}
			\mathrm{dim}~\fr{g}&=\mathrm{dim}~\fr{g}^3+\mathrm{dim}~\fr{g}^2+\mathrm{dim}~\fr{g}^1\\
			&=\frac{1}{2}(n(n+1)+k(n-k)). 
		\end{align*}
		From $\mathrm{dim}~\fr{g}\leq10$ and $k(n-k)\geq1$, 
		\begin{equation}\label{eq}
			n(n+1)\leq10. 
		\end{equation}
		The only integers $n$ that satisfy Inequality \ref{eq} are 2 and 3. In the case of $n=2$, an integer $k$ is only 1. In the case of $(k, n)=(1, 2)$, it satisfies $\mathrm{dim}~\fr{g}=4$. Moreover, $\mathrm{dim}~\fr{g}^3=1$, $\mathrm{dim}~\fr{g}^2=1$ and $\mathrm{dim}~\fr{g}^1=2$. Also, in the case of $n=3$, an integer $k$ is only 2 or 3. In the case of $(k, n)=(1, 3), (2, 3)$, both are $\mathrm{dim}~\fr{g}=8$. Moreover, $\mathrm{dim}~\fr{g}^3=1, 3$, $\mathrm{dim}~\fr{g}^2=2, 2$ and $\mathrm{dim}~\fr{g}^1=5, 3$, respectively. Applying the same procedure to the remaining cases yields the complete Table\ref{parabolic table}

	\end{proof}
\section{Three-step nilpotent Lie algebras obtained by quivers}\label{quiver}
	In this section, we introduce nilpotent Lie algebras obtained by quivers. A quiver is a directed graph where loops and multiple arrows between two vertices are allowed. This method is a way to obtained higher-step nilpotent Lie algebras. 
	\begin{Def}
		A {\it quiver} $Q=(V, E, s, t)$ is a quadruple consisting of two sets $V$, $E$, and two maps $s, t: E\to V$. Elements of $V$ and $E$ are called {\it vertices} and {\it arrows}, respectively. For an arrow $\alpha\in E$, the vertices $s(\alpha)$ and $t(\alpha)$ are called the {\it source} and the {\it target} of $\alpha$, respectively.
	\end{Def}
	Quivers consist of vertices and arrows. Note that the notion of paths is also important, which is defined as follows. 
	\begin{Def}
		Let $Q=(V, E, s, t)$ be a quiver. A {\it path of length $m$} is a sequence $\alpha_1\alpha_2\cdots \alpha_m$, where $\alpha_1, \ldots, \alpha_m\in E$ and they satisfy $t(\alpha_i)=s(\alpha_{i+1})$ for $i=1,\ldots, m-1$.
	\end{Def}
	\begin{Def}[\cite{MT}]
		Let $Q$ be a quiver and $\Path(Q)$ be a set of all paths in $Q$. The path algebra $\fr{n}_Q$ is defined as an $\mathbb{R}$-vector space having $\Path(Q)$ as basis with the following product: for $x$ and $y\in\Path(Q)$, 
		\[
			x\cdot y=
			\begin{cases}
				xy & (xy\in\Path(Q)),  \\
				0 & (\text{otherwise}). 
			\end{cases}
		\]
		The space $\fr{n}_Q$ equipped the Lie bracket $[x, y]=x\cdot y-y\cdot x$ is called the Lie algebra obtained by $Q$. 
	\end{Def}
	If a quiver has cycles, the obtained Lie algebra is infinite dimensional, where a cycle is a paths whose source and target coincide. For a finite quiver without cycles, there exists the maximum of length of paths which is called the length of a quiver. Then, the following holds: 
	\begin{Prop}[\cite{MT}]
		Let $Q$ be a quiver of length $m$. Then the obtained Lie algebra $\fr{n}_Q$ is an $m$-step nilpotent Lie algebra. 
	\end{Prop}
	In this paper, we focus on three-step nilpotent Lie algebras obtained by quivers. The Lie algebra obtained by a connected quiver is the direct sum of the Lie algebras obtained by each connected component. Therefore, we consider connected quiver of length 3. A quiver of length 3 requires at least three arrows. Let $Q_0=(V_0, E_0, s_0, t_0)$ be the following quiver. 
	\begin{figure}[H]
		\includegraphics[scale=0.7]{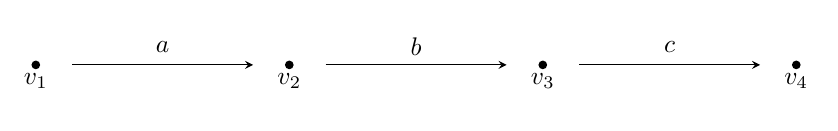}
		\caption{Quiver $Q_0$}
		\label{Q_0}
	\end{figure}
	Let $Quiver^3$ be a set of all connected quivers $Q$ of length 3 satisfying that there do not exist a non-connected quiver $Q'$ such that $\fr{n}_Q\simeq\fr{n}_{Q'}$. For $Q\in Quiver^3$, we assume without loss of generality that $Q$ contains  $Q_0$ as a subquiver. 
	\begin{Lem}\label{properties}
		Let $Q=(V, E, s, t)\in Quiver^3$. Then, the following properties hold. 
		\begin{enumerate}
			\item (loops) $\nexists e\in E$ such that $s(e)=t(e)=v_i$ \quad ($i=1, 2, 3, 4$). 
			\item (cycles) $\nexists e\in E$ such that $s(e)=v_i$, $t(e)=v_j$ and $j<i$. 
			\item (4-step) 
				\begin{enumerate}[(a)]
					\item $\nexists e\in E$ such that $t(e)=v_1$ and $s(e)\in V\setminus V_0$. 
					\item $\nexists e\in E$ such that $s(e)=v_4$ and $t(e)\in V\setminus V_0$. 
				\end{enumerate}
			\item (non-connected) $\nexists e\in E$ such that $s(e)=v_1$ and $t(e)=v_4$. 
		\end{enumerate}
	\end{Lem}
	\begin{proof}
		Assume that there exists $e\in E$ such that $s(e)=t(e)=v_i$ for $i=1, 2, 3, 4$. Then, since $e$ is a loop, $Q$ is not of length 3, which is a contradiction. 
		
		Assume that there exists $e\in E$ such that $s(e)=v_2$ and $t(e)=v_1$. A path $ae$ is a cycle, it satisfies $s(ae)=t(ae)=v_1$. This is a contradiction. The same applies in other cases as well. 
		
		Assume that there exists $e\in E$ such that $t(e)=v_1$ and $s(e)\in V\setminus V_0$.  Since $t(e)=s(a)=v_1$ holds, $eabc$ is a path of length 4. Therefore, since this is contradictory, (a) holds. Similarly, (b) is true. 
		
		Assume that there exists $e\in E$ such that $s(e)=v_1$ and $t(e)=v_4$. From (3), there do not exist $e'\in E$ such that $t(e')=v_1$ or $s(e')=v_4$. Therefore, for any $x\in\Path(Q)$, it satisfies $[e, x]=0$. Let $Q'=(Q\setminus\{e\})\oplus (\{w_1, w_2\}, \{e\}, s'(e)=w_1, t'(e)=w_2)$ be a non-connected quiver. It satisfies $\fr{n}_Q\simeq\fr{n}_{Q'}$. This is a contradiction. 
 	\end{proof}
	\begin{Lem}\label{+2}
		Let $Q=(V, E, s, t)$ be a finite connected quiver without cycles, $e_0\not\in E$, and $\tilde{Q}=(\tilde{V}, E\cup\{e_0\}, \tilde{s}, \tilde{t})$ be a connected quiver that satisfies $V\subset \tilde{V}$, $\tilde{s}|_E=s$, $\tilde{t}|_E=t$ and there do not exist a non-connected quiver $Q'$ such that $\fr{n}_{\tilde{Q}}\simeq\fr{n}_{Q'}$. Then, it satisfies $\#(\Path(\tilde{Q})\setminus\Path(Q))\geq 2$, that is $\mathrm{dim}~\fr{n}_{\tilde{Q}}-\mathrm{dim}~\fr{n}_Q\geq 2$. 
	\end{Lem}
	\begin{proof}
		Since $\#(E\cup\{e_0\})-\#E=1$ holds, it satisfies $\#(\Path(\tilde{Q})\setminus\Path(Q))\geq 1$. Assume that $\#(\Path(\tilde{Q})\setminus\Path(Q))=1$ holds, that is $\Path(\tilde{Q})=\Path(Q)\sqcup\{e_0\}$. Therefore, for any $x\in \Path(\tilde{Q})$, it satisfies $[e, x]=0$. Thus, $\fr{n}_{\tilde{Q}}$ is isomorphic to the Lie algebra obtained by a non-connected quiver. Therefore, since this is contradictory, $\#(\Path(\tilde{Q})\setminus\Path(Q))\geq2$ holds.
	\end{proof}
	\begin{Lem}\label{+3,4}
		Let $Q=(V, E, s, t)\in Quiver^3$ and $e_0\not\in E$. Define two quivers as follows: 
			\begin{itemize}
				\item $Q_1=(V_1, E\cup\{e_0\}, s_1, t_1)$, $s_1(e_0)=v_1$, $t_1(e_0)=v_2$, 
				\item $Q_2=(V_2, E\cup\{e_0\}, s_2, t_2)$, $s_2(e_0)=v_2$, $t_2(e_0)=v_3$, 
			\end{itemize}
			where $s_i|_E=s$ and $t_i|_E=t$ for $i=1, 2$. 
		Then, the following hold: 
		\begin{enumerate}
			\item $\#(\Path(Q_1)\setminus\Path(Q))\geq 3$, 
			\item $\#(\Path(Q_2)\setminus\Path(Q))\geq 4$.  
		\end{enumerate}
	\end{Lem}
	\begin{proof}
		Since $s(b)=s(bc)=v_2$ holds, it satisfies $e_0, e_0b, e_0bc\in \Path(Q_1)\setminus\Path(Q)$. Therefore, $\#(\Path(Q_1)\setminus\Path(Q))\geq 3$ holds. Similarly, since $e_0, ae_0, e_0c, ae_0c\in \Path(Q_2)\setminus\Path(Q)$, it satisfies (2). 
	\end{proof}
	\begin{Lem}\label{rotating}
		Let $Q=(V, E, s, t)$ be a finite quiver without cycles, $Q^-=(V, E, s^-, t^-)$ be a quiver obtained by rotating $Q$, that is $s^-=t$ and $t^-=s$. Then, the obtained Lie algebra $\fr{n}_Q$ is isomorphic to $\fr{n}_{Q^-}$. 
	\end{Lem}
	\begin{proof}
		By the definition of $Q^-$, it satisfies 
		\[
			\Path(Q^-)=\{\alpha_r\cdots \alpha_1\mid \alpha_1\cdots \alpha_r\in\Path(Q)\}. 
		\]
		Define a map $f: \Path(Q)\to\Path(Q^-)$ by
		\[
			f(\alpha_1\cdots\alpha_r)=
			\begin{cases}
				\alpha_r\cdots\alpha_1 & (r=\text{odd}),  \\
				-\alpha_r\cdots\alpha_1 & (r=\text{even}). 
			\end{cases}
		\]
		We prove that $f$ is an isomorphism. First, it is clear that $f$ is a bijection. Next, we show that $f$ preserves the Lie bracket. For paths $x=\alpha_1\cdots\alpha_i$ and $y=\alpha_{i+1}\cdots\alpha_r$, assume that $[x, y]=xy=\alpha_1\cdots\alpha_r$. In the case that $r$ is odd, only one of the lengths of $x$ or $y$ is odd. Therefore, it satisfies
		\begin{align*}
			[f(x), f(y)]&=-[\alpha_i\cdots\alpha_1, \alpha_r\cdots\alpha_{i+1}] \\
			&=\alpha_r\cdots\alpha_1	& (\because t^-(\alpha_{i+1})=s^-(\alpha_i)) \\
			&=f(xy) \\
			&=f([x, y])
		\end{align*}
		In the case that $r$ is even, both lengths of $x$ and $y$ are odd or even. Therefore, in both cases, $f$ preserves the Lie bracket. 
	\end{proof}
	\begin{Thm}
		Any quiver $Q\in Quiver^3$ such that $\mathrm{dim}~\fr{n}_Q\leq 10$ is isomorphic to one of the following. 
	\end{Thm}
	\begin{figure}[H]
		\begin{center}
			\begin{minipage}{.49\columnwidth}
				\centering
				\includegraphics[width=0.9\columnwidth]{Q_0.pdf}
				\caption{Quiver $Q_0$}
				\label{quiverQ_0}
			\end{minipage}
			\begin{minipage}{.49\columnwidth}
				\centering
				\includegraphics[width=0.9\columnwidth]{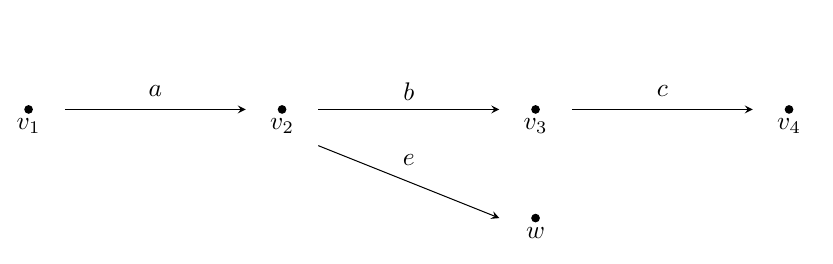}
				\caption{Quiver $Q_1$}
				\label{quiverQ_1}
			\end{minipage}
		\end{center}
	\end{figure}
	\begin{figure}[H]
		\begin{center}
			\begin{minipage}{.49\columnwidth}
				\centering
				\includegraphics[width=0.9\columnwidth]{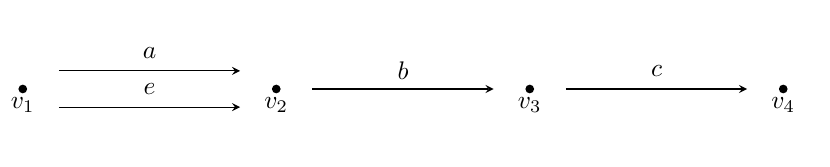}
				\caption{Quiver $Q_2$}
				\label{quiverQ_2}
			\end{minipage}
			\begin{minipage}{.49\columnwidth}
				\centering
				\includegraphics[width=0.9\columnwidth]{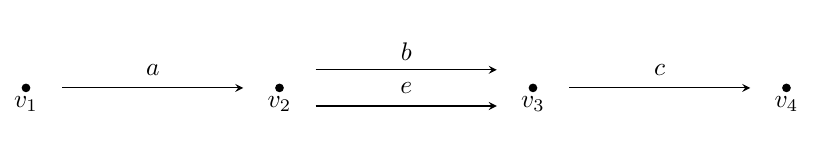}
				\caption{Quiver $Q_3$}
				\label{quiverQ_3}
			\end{minipage}
		\end{center}
	\end{figure}
	\begin{figure}[H]
		\begin{center}
			\begin{minipage}{.49\columnwidth}
				\centering
				\includegraphics[width=0.9\columnwidth]{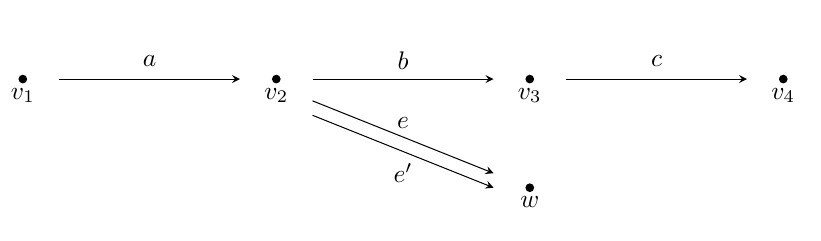}
				\caption{Quiver $Q_4$}
				\label{quiverQ_4}
			\end{minipage}
			\begin{minipage}{.49\columnwidth}
				\centering
				\includegraphics[width=0.9\columnwidth]{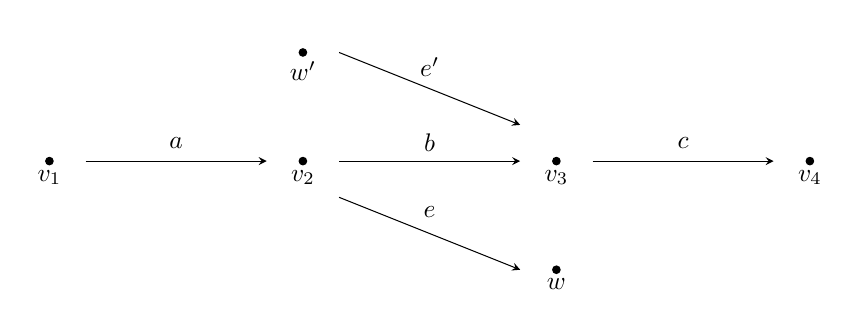}
				\caption{Quiver $Q_5$}
				\label{quiverQ_5}
			\end{minipage}
		\end{center}
	\end{figure}
	\begin{table}[H]
		\begin{center}
			\begin{tabular}{|c|c|c|c|c|}
				\hline
				$\mathrm{dim}~\fr{n}_Q$ & $Q$ & $\#\{\text{paths of length 3}\}$ & $\#\{\text{paths of length 2}\}$ & $\#\{\text{paths of length 1}\}$ \\ \hline
				6 & $Q_0$ & 1 & 2 & 3 \\ \hline
				8 & $Q_1$ & 1 & 3 & 4 \\ \hline
				9 & $Q_2$ & 2 & 3 & 4 \\ \hline
				\multirow{3}{*}{10} & $Q_3$ & 2 & 4 & 4 \\
				\cline{2-5}
				 & $Q_4$ & 1 & 4 & 5 \\
				\cline{2-5}
				 & $Q_5$ & 1 & 4 & 5 \\ \hline
			\end{tabular}
		\end{center}
		\caption{Three-step nilpotent Lie algebras obtained by quivers which are $\leq\mathrm{dim} 10$}
		\label{quiver table}
	\end{table}
	\begin{proof}
		Take any $Q=(V, E, s, t)\in Quiver^3$ and then it satisfies $\# E\geq 3$. In the case of $\# E=3$, the only element of $Quiver^3$ is the quiver $Q_0$. In the case of $\# E=4$, a quiver $Q$ has $Q_0$ as a subquiver, that is $E=\{a, b, c, e\}$. Since $Q$ is connected, $s(e)$ or $t(e)$ is $v_1$, $v_2$, $v_3$ or $v_4$. If we consider all cases regarding the source and target of $e$, the result is as shown in Table \ref{e}, where $w$ in a vertex in $V\setminus V_0$. From Lemma\ref{properties}, it does not satisfy $Q\in Quiver^3$ in the following cases. 
		\begin{itemize}
			\item $s(e)=t(e)=v_i$ for $i=1, 2, 3, 4$, 
			\item $s(e)=t(e)=w$, 
			\item $s(e)=v_i$, $t(e)=v_j$ and $j<i$, 
			\item $t(e)=v_1$ and $s(e)=w$, 
			\item $s(e)=v_4$ and $t(e)=w$, 
			\item $s(e)=v_1$ and $t(e)=v_4$. 
		\end{itemize}
		Additionally, in both case of $s(e)=v_1$ and $t(e)=w$ and case of $s(e)=w$ and $t(e)=v_4$, there exists a non-connected quiver $Q'$ such that $\fr{n}_Q\simeq\fr{n}_{Q'}$. Next, we show that in several cases the obtained Lie algebras are isomorphic. Let $Q_1$ be a quiver where $s(e)=v_2$ and $t(e)=w$, $Q_2$ be a quiver where $s(e)=v_1$ and $t(e)=v_2$, and $Q_3$ be a quiver where $s(e)=v_2$ and $t(e)=v_3$. First, consider quivers isomorphic to $Q_1$. In the case of $s(e)=v_2$ and $t(e)=v_4$, all paths in this quiver are same to $\Path(Q_1)$. Also, in the case of $s(e)=w$ and $t(e)=v_3$, this quiver is isomorphic to $Q_1^-$ where $Q_1^-$ is a quiver obtained by rotating $Q_1$. From Lemma\ref{rotating}, the obtained Lie algebra by this quiver is isomorphic to $\fr{n}_{Q_1^-}$. Moreover, in both case of $s(e)=v_1$ and $t(e)=v_3$ and case of $s(e)=w$ and $t(e)=v_3$,  the set of all paths in each quiver is the same. Base on the above, the Lie algebras obtained in the following four cases are isomorphic. 
		\begin{itemize}
			\item $s(e)=v_2$ and $t(e)=w$\quad($Q_1$), 
			\item $s(e)=v_2$ and $t(e)=v_4$, 
			\item $s(e)=w$ and $t(e)=v_3$, 
			\item $s(e)=v_1$ and $t(e)=v_3$. 
		\end{itemize}
		Similarly, the Lie algebras obtained in the following four cases are isomorphic as well. 
		\begin{itemize}
			\item $s(e)=v_1$ and $t(e)=v_2$\quad($Q_2$), 
			\item $s(e)=w$ and $t(e)=v_2$, 
			\item $s(e)=v_3$ and $t(e)=v_4$, 
			\item $s(e)=v_3$ and $t(e)=w$. 
		\end{itemize}
		Finally, since each dimension is different, the Lie algebras obtained by $Q_1$, $Q_2$ and $Q_3$ are not isomorphic. Therefore, $Quiver^3$ with $\# E=4$ are $Q_1$, $Q_2$ and $Q_3$. 
		\begin{table}[H]
			\begin{center}
				\begin{tabular}{|c|c|c|c|c|c|}
					\hline
					\diagbox{$s(e)$}{$t(e)$} & $v_1$ & $v_2$ & $v_3$ & $v_4$ & $w$ \\ \hline
					$v_1$ & Lem\ref{properties}(1) & $Q_2$ & $Q_1^-$ & Lem\ref{properties}(4) & Lem\ref{properties}(4) \\ \hline
					$v_2$ & Lem\ref{properties}(2) & Lem\ref{properties}(1) & $Q_3$ & $Q_1$ & $Q_1$ \\ \hline
					$v_3$ & Lem\ref{properties}(2) & Lem\ref{properties}(2) & Lem\ref{properties}(1) & $Q_2^-$ & $Q_2^-$ \\ \hline
					$v_4$ & Lem\ref{properties}(2) & Lem\ref{properties}(2) & Lem\ref{properties}(2) &Lem\ref{properties}(1) & Lem\ref{properties}(3) \\ \hline
					$w$ & Lem\ref{properties}(3) & $Q_2$ & $Q_1^-$ & Lem\ref{properties}(4) & Lem\ref{properties}(1) \\ \hline
  				\end{tabular}
			\end{center}
			\caption{The source and target $e$}
			\label{e}
		\end{table}
		Since the dimensions of $\fr{n}_{Q_2}$ and $\fr{n}_{Q_3}$ are 9 and 10, respectively, these quivers cannot have more arrows added to them by Lemma\ref{+2}. Therefore, in the case of $\# E=5$, we only have to consider $Q_1$ with one additional arrow $e'$. For $Q_1$, we similarly consider all cases, as shown in Table\ref{added to Q_1}. Let $Q_4$ be a quiver where $s(e')=v_2$ and $t(e')=w$, $Q_5$ be a quiver where $s(e')=w'$ and $t(e')=v_3$, and $Q_6$ be a quiver where $s(e')=w$ and $t(e')=w'$. Then, the dimension of $\fr{n}_{Q_6}$ is 11. If we examine other quivers isomorphic to $Q_1$ in the same way, we find that only $Q_4$ and $Q_5$ are of dimension $\leq 10$. Therefore, $Quiver^3$ with $\# E=5$ are $Q_4$ and $Q_5$. 
		\begin{table}[H]
			\begin{center}
				\begin{tabular}{|c|c|c|c|c|c|c|}
					\hline
					\diagbox{$s(e')$}{$t(e')$} & $v_1$ & $v_2$ & $v_3$ & $v_4$ & $w$ & $w'$ \\ \hline
					$v_1$ & Lam\ref{properties}(1) & Lem\ref{+3,4} & $Q_5$ &Lam\ref{properties}(4) & Lam\ref{properties}(4) & Lam\ref{properties}(4) \\ \hline
					$v_2$ & Lam\ref{properties}(2) & Lam\ref{properties}(1) & Lem\ref{+3,4} & $Q_4$ & $Q_4$ & $Q_4$ \\ \hline
					$v_3$ & Lam\ref{properties}(2) & Lam\ref{properties}(2) & Lam\ref{properties}(1) & Lem\ref{+3,4} & Lem\ref{+3,4} & Lem\ref{+3,4} \\  \hline
					$v_4$ & Lam\ref{properties}(2) & Lam\ref{properties}(2) & Lam\ref{properties}(2) & Lam\ref{properties}(1) &Lam\ref{properties}(3) & Lam\ref{properties}(3) \\ \hline 
					$w$ & Lam\ref{properties}(2) & Lam\ref{properties}(2) & Lam\ref{properties}(3) & $Q_6$ & Lam\ref{properties}(1) & $Q_6$ \\ \hline
					$w'$ & Lam\ref{properties}(3) & Lam\ref{+3,4} & $Q_5$ & Lam\ref{properties}(4) & Lam\ref{properties}(4) & Lam\ref{properties}(1) \\ \hline   
				\end{tabular}
			\end{center}
			\caption{All cases where one arrow is added to $Q_1$}
			\label{added to Q_1}
		\end{table}
	\end{proof}
\section*{Acknowledgements} 
	The author is deeply grateful to Professors Ali Baklouti and Hideyuki Ishi, the organizers of the 8th Tunisian--Japanese Conference, for providing a valuable opportunity for stimulating discussions. This work originated from fruitful exchanges that took place during the conference. 
	
	This work was partly supported by MEXT Promotion of Distinctive Joint Research Center Program JPMXP0723833165 and Osaka Metropolitan University Strategic Research Promotion Project (Development of International Research Hubs).


\begin{thebibliography}{99} 
	\bibitem [K] {K}{Knapp,~Anthony~W.}, {``Lie groups beyond an introduction''}, {Progr. Math.~\textbf{140}. Birkh\"auser Boston, Inc., Boston, MA (2002)}. 
	\bibitem [MT]{MT}{Mizoguchi,~F. and Tamaru,~H.}, {``Nilpotent Lie algebras obtained by quivers and Ricci solitons''}, {Adv. Math.~\textbf{480}~(2025), 110464 [22 pp]}. 
	\bibitem [T1] {T1}{Tamaru,~H.}, {``The local orbit types of symmetric spaces under the actions of the isotropy subgroups''}, {Differential Geom. Appl. \textbf{11} (1999), no.~1, 29--38}. 
	\bibitem [T2] {T2}{Tamaru,~H.}, {``Parabolic subgroups of semisimple Lie groups and Einstein solvmanifolds''}, {Math. Ann.~\textbf{351} (2011)}, {no.~1, 51--66}. 
\end{thebibliography}
\end{document}